\documentclass[twocolumn]{autart}    

\usepackage{cite}
\usepackage{psfrag}
\usepackage{amsfonts}
\usepackage{amssymb}
\usepackage{enumerate}
\usepackage{appendix}
\usepackage{psfrag,epsfig,graphicx,ae}
\usepackage{color}

\usepackage{verbatim}

\usepackage[applemac]{inputenc}
\usepackage{graphicx} 
\usepackage{amsmath}
\usepackage{graphics}
\usepackage{graphicx,comment,}          

\usepackage{amsmath} 
\usepackage{amssymb}  
\usepackage{cite}
\usepackage{tikz}
\usepackage{epstopdf}
\usepackage{epsfig}
\usepackage{caption}
\usepackage{subcaption}
\usepackage{lscape}
\usepackage{psfrag}
\usepackage{mathrsfs}
\usepackage{tikz}
\usetikzlibrary{shapes,arrows}
\usetikzlibrary{arrows, decorations.markings}

\usepackage{eurosym}
\usepackage{amstext} 
\DeclareRobustCommand{\officialeuro}{%
  \ifmmode\expandafter\text\fi
  {\fontencoding{U}\fontfamily{eurosym}\selectfont e}}

\newcounter{teocount}
\newcounter{propcount}
\newcounter{corcount}
\newcounter{remcount}
\newcounter{defcount}
\newcounter{asscount}
\newcounter{excount}
\newcounter{probcount}
\newcounter{lemcount}

\newtheorem{remm}[remcount]{Remark}
\newtheorem{assumption}[asscount]{Assumption}
\newtheorem{definition}[defcount]{Definition}
\newtheorem{proposition}[propcount]{Proposition}
\newtheorem{theorem}[teocount]{Theorem}
\newtheorem{corollary}[corcount]{Corollary}
\newtheorem{ex}[excount]{Example}
\newtheorem{problem}[probcount]{Problem}
\newtheorem{lemma}[lemcount]{Lemma}

\newenvironment{remark}{\begin{remm}\rm }{\hfill \hspace*{1pt} \hfill $\lrcorner$\end{remm}}
\newenvironment{example}{\begin{ex}\rm }{\hfill \hspace*{1pt} \hfill $\circ$\end{ex}}
\newenvironment{proof}{{\em Proof. }}{\hfill \hspace*{1pt} \hfill $\blacksquare$}

\newcommand{\real}{{\mathbb R}}
\newcommand{\nt}{{\mathbb N}}

\begin{document}

\begin{frontmatter}

\title{Barrier Functionals for the Analysis of Complex Systems: \\An Optimization-Based Approach\thanksref{footnoteinfo}} 

\thanks[footnoteinfo]{M. Ahmadi is supported by  a Clarendon Scholarship and  the Sloane-Robinson Scholarship. A.~Papachristodoulou is supported in part by EPSRC projects EP/J012041/1, EP/M002454/1 and EP/J010537/1. A preliminary version of this paper was presented at the 2015 American Control Conference, July 1-3, Chicago, IL, USA. Corresponding author M.~Ahmadi. Tel. +44 1865 2 83036.
Fax +44 1865 273010.}

\author[Rome]{Mohamadreza Ahmadi}\ead{mohamadreza.ahmadi@eng.ox.ac.uk},    
\author[france]{Giorgio Valmorbida}\ead{giorgio.valmorbida@l2s.centralesupelec.fr},               
\author[Rome]{Antonis Papachristodoulou}\ead{antonis@eng.ox.ac.uk}  

\address[Rome]{Department of Engineering Science,  University of Oxford, Parks Road, Oxford, OX1 3PJ, UK.}             
\address[france]{ Laboratoire des Signaux et Syst\`emes, CentraleSup\'elec, CNRS, Univ. Paris-Sud, Universit\'e Paris-Saclay, 3~Rue~Joliot-Curie, Gif sur Yvette
91192, France.}        

\begin{keyword}                           
Set Avoidance,  Moment Estimation, Sum of Squares, Feynman-Kac Formula, Integral Inequalities.              
\end{keyword}                             

\begin{abstract}                          
We propose a methodology to address two  analysis problems concerning complex systems, namely bounding state functionals of stochastic differential equations (SDEs) and verifying set avoidance of systems described by partial differential equations (PDEs).  The proposed method  is based on barrier functionals, which are functionals of the states of the studied systems. The proposed method does not require the approximation of  solutions nor the stability of trajectories. In the case of SDEs, the formulation relies on a generalized version of the Feynman-Kac formula and results in moments bounds for nonlinear SDEs. Furthermore, we show that the analysis problems can be cast as optimization problems and can be solved by semi-definite programming.
\end{abstract}

\end{frontmatter}

\section{Introduction}

\label{sec:intro}


\subsection{Motivation:}

The dynamics of  physical and engineered systems are often modeled by differential equations. Solving  these differential equation models  is not straightforward in general, especially in the case of nonlinear differential equations. Hence, numerical methods, such as the finite-difference method, the finite-element method~\cite{HOLLAND19743} and the Euler-Maruyama method~\cite{Burrage2000171}, are used to approximate the solutions.  Yet, for  significant classes of systems, approximating the solutions is  too computationally demanding, especially if there is parameter/initial condition uncertainty. Therefore, researchers have turned to alternative methods for studying the properties of these systems. In the control literature, the renowned Lyapunov theorem for stability is one  such approach, which answers the stability question without solving the differential equation~\cite{citeulike1443928}. Verifying  input-output properties using dissipativity theory  is another example of a method which averts computing the solutions for different initial conditions, inputs and etc~\cite{ref1}.




In some applications, rather than the solutions, we may only be interested in  a functional output of the solutions of an underlying system.  For example, for systems described by SDEs, we might be interested in evaluating the moments or the time-integral of the moments at a particular point in time. As an example, functionals of the state of the stochastic models for asset prices describe the price of an option~\cite{Mer73}.  For nonlinear SDEs, modeling the dynamics of the statistical moments is not trivial, because the dynamics of the lower-order moments depend on the higher-order moments.  This leads to the moment closure problem, which is based on closing the dynamics by expressing the higher-order moments as a function of the lower-order moments. The problem has been studied in the context of biological applications~\cite{7226883,KCMG05} and, in particular, biochemical reaction networks~\cite{4762250,5605237}. However, existing methods only provide approximate solutions to the moment dynamics.

In the case of systems described by partial differential equations (PDEs), the review article~\cite{Babuska1994} presents a number of functional approximation problems in structural mechanics. The far-field pattern in electromagnetics and acoustics~\cite{Monk98theadaptive} and energy release rate in elasticity theory~\cite{Xuan2006430} are both functionals of the solutions to the governing differential equations. Perhaps a more interesting example is in fluid mechanics, i.e., lift and drag forces acting on an airfoil surrounded by  incompressible flow (described by the Navier-Stokes equations) are defined as functionals of pressure and shear forces over the surface of the airfoil~\cite{Machiels1999}.

In this paper, we formulate the functional  bounding problem as a \emph{set avoidance problem}. The set avoidance problem can also be regarded as a tool for \textit{safety verification}. We call the set that is avoided by the functional output trajectories, for all time or for finite intervals of time, an \emph{undesirable set}.  Regarding set avoidance strategies for ordinary differential equation (ODE) systems, methods based on approximations of the reachable set are considered in~\cite{Kurzhanski2000201} for linear systems and in~\cite{TMBO03} for nonlinear systems. One method to study the set avoidance problem, which does not require the approximation of reachable sets, uses \textit{barrier certificates}. Barrier certificates \cite{P06} were introduced for model invalidation of ODEs with polynomial vector fields defined on semi-algebraic sets.  This method has been used to address safety verification of nonlinear and hybrid systems~\cite{PJP07}, safety analysis of time-delay systems~\cite{1582846},   and model invalidation of complex biological networks~\cite{AP09}.
    Moreover, compositional barrier certificates and converse results were studied in \cite{SWP12} and \cite{7236867}, respectively.
    Computation of barrier certificates using interval analysis was studied in~\cite{BCDK14}.
    Controller synthesis methods based on barrier certificates were also described in~\cite{WA07}. 
    
\subsection{Contribution:}

The main contributions of this paper are threefold:

\begin{itemize}
\item We propose a method for bounding state functionals of systems described by stochastic differential equations. The method is based on a generalized version of the Feynman-Kac formula, which describes the backward dynamics of a cost functional of the moments of an SDE. We use barrier functionals to compute these state functionals. We demonstrate that if the barrier functional satisfies two inequalities along the solutions of the backward dynamics, then  we can infer bounds on the cost functional of the moments.
\vspace{5mm}
\item We present a method for verifying set avoidance of  systems described by PDEs. In this case, the considered sets are subsets of Hilbert spaces rather than only subsets of Euclidean spaces in the case of systems described by ODEs. This method is also based on barrier functionals. Provided that the barrier functional satisfies a set of inequalities along the solutions of the PDE, we can conclude that the solutions avoid an undesirable set for all time or at some specific time instant.
\vspace{5mm}
\item We show that, if some assumptions are made regarding the polynomial dependence on the problem data and if the subsets of Hilbert spaces are defined by integral inequalities with polynomial integrands, the results in the paper can be implemented using semi-definite programming. In this regard, the results in~\cite{VAP16} and in~\cite{Ahmadi2016163} consider integral inequalities with time independence. 
\end{itemize}

A preliminary application of the barrier functionals for upper bound estimation of functional outputs of  PDEs with polynomial data was discussed in~\cite{AVP15}, where we showed that the output functional approximation problem can be cast in a set avoidance framework. The output functional estimation problem consist of a particular instance of the set of problems that can be studied using barrier functionals as discussed in this paper.  Moreover, the proposed computational method in ~\cite{AVP15} is conservative when compared with the approach we propose in this paper in Section~\ref{sec:computational}. 
 
\subsection{Outline:}

{
  The rest of the paper is organized as follows. In the next section, we review some mathematical preliminaries to systems described by SDEs and PDEs. In Section~\ref{sec:mainresultsss}, we propose a framework based on barrier functionals to bound state functionals of SDEs and to check set avoidance for a PDE system. In Section~\ref{sec:computational}, under the assumption of polynomial data, we propose a method based on Semi-Definite Programming (SDP) to verify a class of integral inequalities with constraints, which will be used to compute the barrier functionals. The proposed method is illustrated with  examples in Section~\ref{sec:examples}. Finally, we conclude the paper in Section~\ref{sec:conclusion}.
}
\vspace{0.5 cm}

\textbf{Notation:} {The $n$-dimensional Euclidean space is denoted by $\mathbb{R}^n$, the set of nonnegative reals by $\mathbb{R}_{\ge0}$, and the set of nonpositive reals by $\mathbb{R}_{\le0}$. The $n$-dimensional set of positive integers is denoted by $\mathbb{N}^n$, and the $n$-dimensional space of non-negative integers is denoted by $\mathbb{N}^n_{\ge0}$. 
We use $M^{T}$ to denote the transpose of matrix $M$ and $\mathrm{Tr}\{M\}$ is the trace of the square matrix $M$. 
A domain $\Omega$ is an open subset of $\mathbb{R}^n$ and the boundary of $\Omega$ is denoted $\partial \Omega$. The ring  of polynomials on real variables $x \in \real^n$ and $y \in \real^m$  is  denoted $\mathcal{R}[x,y]$.
The space of $k$-times continuous differentiable functions defined on $\Omega$ is denoted by $\mathcal{C}^k(\Omega)$ and the space of $\mathcal{C}^k(\Omega)$ functions mapping to a set $\Gamma$ is denoted \mbox{$\mathcal{C}^k(\Omega ; \Gamma)$}. For a multivariable function $f(x,y)$, we use  $f(x,\cdot)\in \mathcal{C}^k[x]$ to denote the $k$-times continuous differentiability of $f$ with respect to variable $x$.  If $p\in \mathcal{C}^1(\Omega)$, then $\partial_x p$ denotes the derivative of $p$ with respect to variable $x \in \Omega$.  In addition, we adopt Schwartz's  multi-index notation. For $u \in \mathcal{C}^{k} (\Omega ; \mathbb{R}^m)$, $\Omega \in \real^{n}$, $\alpha \in \nt_0$, defining matrix $A \in \nt_0^{\sigma \times n}$, $\sigma = \frac{(n+\alpha)!}{n!\alpha!}$ (denote its $i$th row $A_i$) which contains a set of ordered elements satisfying $\sum_{j}A_{ij}\leq \alpha$, we have
$$
D^{\alpha}u := \left(u_1, \partial_x{u_{1}},  \ldots, {\partial_x^{A_\sigma}{u_{1}}},  \ldots, u_m,{\partial_x u_{m}}, \ldots, {\partial^{A_\sigma}_x{u_{m}}} \right), $$
where $\partial_x^{A_{i}}(\cdot)=\partial_{x_1}^{A_{i1}}(\cdot) \cdots \partial_{x_n}^{A_{in}}(\cdot)$. We use the same multi-index notation to denote a vector of monomials up to degree $\alpha$ on a variable $x$ as $\zeta^\alpha$. For instance, for $x\in \real^2$, $\zeta^2(x) = (1,x_1,x_2,x_1^2, x_1x_2, x_2^2)$. For a function $f \in \mathcal{C}^1(\Omega)$ and $g \in \mathcal{C}^2(\Omega)$, $\nabla f$ denotes the gradient vector, $\nabla^2 g$ denotes the Hessian matrix and $\Delta g$ is the Laplacian operator. A Hilbert space of functions defined over the domain $\Omega$ with the norm
$\| u \|_{ \mathcal{W}^p_\Omega} = \left( \int_\Omega \sum_{i=0}^p (\partial^i_x u)^{T} (\partial^i_x u)  \,\, \mathrm{d}x   \right)^{\frac{1}{2}}$ is denoted  $ \mathcal{W}^p_{\Omega}$.
 By $f \in \mathcal{L}^2(\Omega;\Gamma)$, we denote a square integrable function mapping $\Omega\subseteq \mathbb{R}^n$ to $\Gamma \subseteq \mathbb{R}^m$. Also, $Dom(\mathcal{A})$ and $Ran(\mathcal{A})$ denote the domain and range of the operator $\mathcal{A}$, respectively.  For a random variable $X$, $E\left[ X\right]$ denotes its expected value. 
}

\section{Preliminaries}

In this section, we review some of the preliminary results and definitions.

Let $T>0$ and let $(\Gamma, \mathcal{J}, \{\mathcal{J}_s\}_{s \ge 0}, \mathrm{P})$ be a complete and right-continuous  filtered probability space, where $\Gamma$ is a sample space,  $\{\mathcal{J}_s\}_{s \ge 0}$ with $\mathcal{J}_s \subseteq \mathcal{J}$ for each $s$  is a filteration of the $\sigma$-algebra $\mathcal{J}$, and $\mathrm{P}$ is the probability measure function. The filteration $\mathcal{J}_s$ satisfies the usual conditions, i.e., it is complete, right continuous ($J_s=\cap_{r>s}\mathcal{J}_r$) and $\mathcal{J}_0$ contains all $\mathrm{P}$-negligible events.  Consider the following SDE
\begin{equation} \label{eq:SDE}
\begin{cases} \mathrm{d}X(s) = b(s,X(s)) \mathrm{d}s + \sigma(s,X(s)) \mathrm{d}W(s), & s \in [t,T], \\
 X(t)=x,\end{cases}
\end{equation}
where $x$ is a $\mathcal{J}_t$-measurable random variable,   $X(s) \in \mathbb{R}^n$ denotes the states and $W(s) \in \mathbb{R}^m$ is an $m$-dimensional standard $\{\mathcal{J}_s\}_{s \ge 0}$-Wiener process starting at $t$ (i.e., $W(t)=0$). A solution to SDE~\eqref{eq:SDE} satisfies the following  integral equation almost surely
\begin{multline*}
X(s)= X(0)+ \int_0^{s} b(\theta,X(\theta)) \,\, \mathrm{d}\theta \\+\int_0^{s} \sigma(\theta,X(\theta)) \,\, \mathrm{d}W(\theta),~s \in [0,T]
\end{multline*}
which is the sum of a Lebesgue integral and an It\^{o} integral~\cite[Chapter 3]{Ok03}. 

Also, consider the following backward in time PDE
\begin{eqnarray} \label{eq:FK}
- \partial_t u(t,x) &=& \frac{1}{2} \mathrm{Tr} \big \{ \sigma(t,x) \sigma^{T}(t,x) \nabla^2 u(t,x) \big\} \nonumber  \\ &&+ b^{T}(t,x) \nabla u(t,x) + c(t,x) u(t,x) \nonumber \\ &&+ h(t,x),\quad \quad \quad(t,x) \in [0,T) \times \mathbb{R}^n, \nonumber \\
u(T,x) &=& f(x) \in  \mathcal{U}_T,
\end{eqnarray}
where $\mathcal{U}_T$ is the set of terminal conditions.

\begin{assumption}
\label{ass:1}
 The maps $b: [0,T] \times \mathbb{R}^n \to \mathbb{R}^n$ and $\sigma: [0,T] \times \mathbb{R}^n \to \mathbb{R}^{n\times m}$, $c,h:[0,T] \times \mathbb{R}^n \to \mathbb{R}$, and $f:\mathbb{R}^n \to \mathbb{R}$ are uniformly continuous, $c$ is bounded, and there exist a constant $L>0$ such that for $\phi = b,\sigma, c,h,$
\begin{equation*}
\begin{cases}
| \phi(t,x) - \phi(t,\hat{x}) | \le L |x -\hat{x} |, & \forall t \in [0,T], x,\hat{x} \in \mathbb{R}^{n}\\
|\phi(t,0)|\le L, & \forall t \in [0,T].
\end{cases}
\end{equation*}
\end{assumption}

\begin{remark}
If  Assumption~\ref{ass:1} holds, there exists a unique $\{\mathcal{J}_s\}_{s \ge 0}$-adapted continuous process $X(s),~s\ge0$ that is a unique strong solution to the SDE~\eqref{eq:SDE}~(see~\cite[Definition~6.2]{YZ99}).
\end{remark}

We recall the following result \cite[Theorem 4.1]{YZ99} which generalizes the Feynman-Kac formula.

\begin{theorem} \label{thm:FK}
Let  Assumptions I hold. Then, \eqref{eq:FK} admits a unique viscosity solution given by
\begin{multline} \label{eq:Solution}
u(t,x) = E\left[ \int_t^T h(s,X(s)) e^{-\int_t^s c(r,X(r)) \mathrm{d}r} \mathrm{d}s \right.\\ \left.
+f(X(T)) e^{-\int_t^T c(r,X(r)) \mathrm{d}r} \mid X(t)=x \right],\\~(t,x) \in [0,T]\times \mathbb{R}^n,
\end{multline}
where $X$ is the unique strong solution of SDE~\eqref{eq:SDE}. In addition, if~\eqref{eq:SDE} admits a classical solution, then \eqref{eq:Solution} is a classical solution of~\eqref{eq:FK}.
\end{theorem}

\begin{remark}
The function $u(t,x)$ in~\eqref{eq:Solution} specifies the value of the functional 
$$
E\left[ \int_t^T h(s,X(s)) e^{-\int_t^s c(r,X(r)) \mathrm{d}r} \mathrm{d}s+f(X(T)) \right],
$$
for a solution of SDE~\eqref{eq:SDE} starting at time $t$ with initial condition $x$.
\end{remark}

The above theorem relates the solutions of the SDE~\eqref{eq:SDE} to the solution of the backward PDE~\eqref{eq:FK} through functional~\eqref{eq:Solution}. The functional given in~\eqref{eq:Solution} encompasses a rich class of state functionals of SDE~\eqref{eq:SDE}. For instance, for $c=0$, $h(s,X(s))=X^2(s)$ and $f(X(T))=X^2(T)$, 
$$
u(0,x)= E\left[ \int_0^T X^2(s)\mathrm{d}s+ X^2(T) \mid X(0)=x \right],
$$
represents the finite-time cost functional with terminal value of the second moment of the solutions to SDE~\eqref{eq:SDE}.

In order to specify boundary conditions for the backward PDE~\eqref{eq:FK} and limit $x$ to a bounded domain $\Omega \subset \mathbb{R}^d$, it is possible to formulate a theorem which relates SDE~\eqref{eq:SDE} to a PDE with boundary conditions. Let $\Omega \subseteq \mathbb{R}^d$ be a bounded domain with $\mathcal{C}^1$ boundary $\partial \Omega$. Define 
\begin{equation*}
\tau :=  \inf \left\{  s \in [t,T] \mid X(s) \notin \Omega \right\}.
\end{equation*}
Then, we can obtain the following result \cite[Theorem 4.2]{YZ99}.
\begin{theorem} \label{thm:FKBC}
Consider \eqref{eq:FK} with boundary conditions $u \mid_{\partial \Omega} = \psi(t,x)$ and SDE~\eqref{eq:SDE}. Let Assumptions I hold with the  functions $b,\sigma,c,h$ defined on $[0,T] \times \Omega$ and $f$ defined on $\Omega$. Let
\begin{equation*}
\Psi(t,x) = \begin{cases} f(x), & (t,x) \in [0,T] \times \Omega, \\ \psi(t,x), & (t,x)\in[0,T] \times \partial \Omega,  \end{cases}
\end{equation*}
be a continuous function on $([0,T] \times \Omega) \cup ([0,T] \times \partial \Omega)$. 
 Then, \eqref{eq:FK} with boundary conditions $u \mid_{\partial \Omega} = \psi(t,x)$ admits a unique viscosity solution given by
\begin{multline} \label{eq:Solution2}
u(t,x) = E\left[ \int_t^\tau h(s,X(s)) e^{-\int_t^s c(r,X(r)) \mathrm{d}r} \mathrm{d}s \right. \\ \left. +\Psi(\tau,X(\tau)) e^{-\int_t^\tau c(r,X(r)) \mathrm{d}r} \mid X(t)=x \right],\\~(t,x) \in [0,T]\times \Omega,
\end{multline}
where $X$ is the unique strong solution of SDE~\eqref{eq:SDE}. In addition, if~\eqref{eq:SDE} admits a classical solution, then \eqref{eq:Solution2} gives a classical solution.
\end{theorem}

\begin{remark}
Consider equation~\eqref{eq:FK} with $c,h=0$. This is equivalent to the \emph{backward Kolmogorov equation}~\cite[Chapter 4]{Risken89}, for which we have
$$
u(t,x) = E\left[ f\left(X(T)\right) \mid X(t)=x \right],
$$
where $X$ is described by~\eqref{eq:SDE}.
\end{remark}

We also study conditions for a class of forward PDE systems. Let $ \mathcal{U}$ be a Hilbert space, consider the following differential equation
\begin{equation}  
\label{eq:pde}
\begin{cases}
\partial_t u(t,x)= \mathscr{F} u(t,x), \quad x\in \Omega \subset\real^n,~t \in [0,t_0],\\
y(t) = \mathscr{H} u(t,x)\\
u(0,x)=u_0(x) \in \mathcal{U}_0\subset Dom(\mathscr{F})\\
u\in \mathcal{U}_b  
\end{cases}
\end{equation}
$\mathcal{U}_b$ is the subspace of $\mathcal{U}$ defined by the boundary values of variable $u$, $\mathscr{H}: \mathcal{U} \rightarrow \real$ and $Dom(\mathscr{H})\supseteq \mathcal{U}$, the state-space of system~\eqref{eq:pde}. It is assumed that~\eqref{eq:pde} is well-posed (see Appendix~A).

 We call the set  
$$
 \mathcal{Y}_u = \big \{ u \in  \mathcal{U} \mid \mathscr{H}  u \le 0  \big\},
$$
the \emph{undesirable set}. As an example of system~\eqref{eq:pde}, consider the following system, $x\in (0,1),~t \in [0,t_0],$
\begin{equation*}  
\begin{cases}
\partial_t u(t,x)= \partial_x^2 u(t,x) -u(t,x) \partial_xu(t,x) , \\
y(t) = 5 - \int_\Omega u^{2}(t,\theta) \mathrm{d} \theta -\partial_xu(t,1)\\
 \mathcal{U}_0 = \left\lbrace u_0 \in \mathcal{L}_2 \mid \|u_0\|_{\mathcal{L}_2}\leq 1; \partial_xu_{0}\leq 0 \right\rbrace \\
\mathcal{U}_b  = \left\lbrace u \in \mathcal{W}^1 \mid \left[\begin{array}{cccc} 1& 0 & 0 & 0\\ 0 & 0 & 1& 0  \end{array}\right]  \left[\begin{array}{c} u(t,0) \\ \partial_x u(t,0) \\ u(t,1) \\ \partial_x u(t,1) \end{array}\right] = 0 \right\rbrace.
\end{cases}
\end{equation*}

In the following section, we present conditions to obtain certificates that trajectories starting in the set $\mathcal{U}_0$ \emph{avoid} the set $\mathcal{Y}_u$. 

\section{Barrier Functionals for the Analysis of SDEs and PDEs} \label{sec:mainresultsss}


We present a method to bound output functionals of SDEs an PDEs. We relate the problem of bounding the output functionals to the problem of \emph{set avoidance}, namely, the problem of verifying whether, for a set of initial condition, a set in the state space is avoided. Such a formulation also allows to obtain performance estimates whenever the output functional to be bounded represents a performance index. 


Consider the following properties of trajectories related to an initial set $\mathcal{U}_0$ and an undesirable set $\mathcal{Y}_u$

\begin{definition} [Set Avoidance at Time $t_0$] \label{def:safe}
Let $u\in~\mathcal{U}$. For a set $ \mathcal{U}_0 \subseteq  \mathcal{U}$ ($ \mathcal{U}_T \subseteq  \mathcal{U}$), an undesirable set $ \mathcal{Y}_u$, satisfying $ \mathcal{U}_0 \cap \mathcal{Y}_u = \emptyset$ ($ \mathcal{U}_T \cap \mathcal{Y}_u = \emptyset$), and a positive scalar $t_0$, system~\eqref{eq:pde} (system~\eqref{eq:FK})   \emph{avoids $ \mathcal{Y}_u$ at time $t_0$}, if, for all $u(0,x) \in \mathcal{U}_0$ ($u(T,x) \in \mathcal{U}_T$), the solutions $u(t,x)$ of system~\eqref{eq:pde} (system~\eqref{eq:FK}) satisfy $y(t_0) \notin  \mathcal{Y}_u$.
\end{definition}

\begin{definition} [Set Avoidance] \label{def:safe2}
System~\eqref{eq:pde} (system~\eqref{eq:FK})  {avoids $\mathcal{Y}_u$}, if it avoids $\mathcal{Y}_u$ in the sense of Definition~\ref{def:safe} for all $t_0\in[0,\infty)$.
\end{definition}

We are interested in solving the following problem:

\begin{problem}
Given sets $\mathcal{Y}_u$ and $ \mathcal{U}_0$ ($\mathcal{U}_T$) and $t_0>0$, verify that system~\eqref{eq:pde} (system~\eqref{eq:FK}) avoids $ \mathcal{Y}_u$ at time $t_0$.
\end{problem}

To this end, we define the \emph{Barrier Functional} 
\begin{equation} \label{eq:barrierfunctional}
B(t,u) = \mathscr{B}(t)u,
\end{equation}
 where $\mathscr{B}(t) : Dom(\mathscr{B}) \to \mathbb{R}$.

\subsection{ Bounding State Functionals of SDEs}

In the following, we propose a method based on barrier functionals to find bounds on the solutions of the PDE~\eqref{eq:FK} and therefore state functionals of~\eqref{eq:SDE}. 

\begin{theorem} [Set Avoidance for Backward Systems] \label{thm:fwd2}
Given a set of terminal conditions \begin{equation}\label{t543} \mathcal{U}_T =  \left\{ u \in \tilde{\mathcal{U}} \mid u(T,x)=f(x) \right\},\end{equation} an undesirable set~$ \mathcal{Y}_u$ such that $ \mathcal{U}_T \cap  \mathcal{Y}_u = \emptyset$, and $t_0\in [0,T]$, if there exists a barrier functional $B(t,u(t,x)) \in \mathcal{C}^1[t]$ as in \eqref{eq:barrierfunctional}
 such that the following inequalities hold
 \begin{subequations}
\label{eq:con}
\begin{multline} \label{eq:con111}
B(t_0, u(t_0,x)) - B(T, u(T,x)) >0, \\ \forall u(t_0,x) \in  \mathcal{Y}_u,~\forall u(T,x) \in  \mathcal{U}_T
\end{multline}
\vspace{-0.7cm}
\begin{equation} \label{eq:con222}
 \frac{\mathrm{d}B(t,u(t,x))}{\mathrm{d}t} \ge 0,~ \forall t\in[0,T],~\forall u \in  \tilde{\mathcal{U}},
 \end{equation}
\end{subequations}
along the solutions of \eqref{eq:FK}, then the solutions $u(t,x)$ of~\eqref{eq:FK} avoid $\mathcal{Y}_u$ at time $t_0 \in [0,T]$.
\end{theorem}

\begin{proof}
We prove the theorem by contradiction. Assume that at time $t_0 \in [0,T]$, there exists a solution $u(t,x)$ to \eqref{eq:FK} with $u(T,x)\in \mathcal{U}_T$ that satisfies $u(t_0,x) \in \mathcal{Y}_u$. Then, from~\eqref{eq:con111}, we have
\begin{equation} \label{fsfsdfsvccv}
B(t_0, u(t_0,x)) - B(T, u(T,x)) >0.
\end{equation}
On the other hand, inequality~\eqref{eq:con222} implies that for all $t \in [0,T]$, it holds that $ \frac{\mathrm{d}B(t,u(t,x))}{\mathrm{d}t} \ge 0$. Integrating from $t$ to $T$  both sides the latter inequality yields
$$
\int_t^{T}  \frac{\mathrm{d}B(t,u(t,x))}{\mathrm{d}t} \,\, \mathrm{d}t = B(T,u(T,x))-B(t,u(t,x)) \ge 0.
$$
Since $t,t_0 \in [0,T]$, this contradicts \eqref{fsfsdfsvccv}. Therefore, there is no solution to \eqref{eq:FK} that satisfies $u(t_0,x) \in \mathcal{Y}_u$.
\end{proof}

Next, we describe how Theorem~\ref{thm:fwd2} can be used to find bounds on the state functionals of SDE~\eqref{eq:SDE}. The method relies on an appropriate definition for the undesirable set $\mathcal{Y}_u$ and  a suitable optimization problem.

\begin{corollary}\label{cor:bwd}
Consider PDE~\eqref{eq:FK} and SDE~\eqref{eq:SDE}. Let
\begin{equation} \label{eq:rre33}
\mathcal{Y}_u = \left \{ u \in  \tilde{\mathcal{U}} \mid u(t_0,x) > \gamma  \right\}.
\end{equation}
If there exists a barrier functional $B(t,u(t,x)) \in \mathcal{C}^1[t]$ such that inequalities~\eqref{eq:con111} and~\eqref{eq:con222} are satisfied along the solutions of \eqref{eq:FK}, then it holds that
\begin{multline} \label{cvncmvcvm}
E\left[ \int_{t_0}^T h(s,X(s)) e^{-\int_{t_0}^s c(r,X(r)) \mathrm{d}r} \mathrm{d}s \right. \\ \left.
+f(X(T)) e^{-\int_{t_0}^T c(r,X(r)) \mathrm{d}r} \mid X(t_0)=x \right] \le \gamma.
\end{multline}
\end{corollary}
\begin{proof}
If there exists a barrier functional $B(t,u(t,x)) \in \mathcal{C}^1[t]$ such that inequalities~\eqref{eq:con111} and~\eqref{eq:con222} are satisfied along the solutions of \eqref{eq:FK}, from Theorem~\eqref{thm:fwd2}, we can infer that $u(t_0,x) \notin \mathcal{Y}_u$ with $\mathcal{Y}_u$ as described by~\eqref{eq:rre33}. Thus, $u(t_0,x) \le \gamma$. From Theorem~\ref{thm:FK}, we have
\begin{multline}
u(t_0,x) = E\left[ \int_{t_0}^T h(s,X(s)) e^{-\int_{t_0}^s c(r,X(r)) \mathrm{d}r} \mathrm{d}s \right. \\ \left. +f(X(T)) e^{-\int_{t_0}^T c(r,X(r)) \mathrm{d}r} \mid X(t_0)=x \right]. \nonumber
\end{multline}
Therefore, $u(t_0,x) \le \gamma$ implies that \eqref{cvncmvcvm} holds.
\end{proof}
In order to find the minimum  $\gamma$ in \eqref{cvncmvcvm}, i.e., the upper bound to the state functional, we solve the following optimization problem
\begin{equation} \label{eq:optm}
\displaystyle{\text{minimize}_{B(t,u(t,x))}} \gamma \quad  \text{\textit{subject to} \quad \eqref{eq:con}}.
\end{equation}

\begin{remark}
In~\cite{7040110}, an SDP-based method for bounding the moments of continuous-time Markov chains, based on  the Foster-Lyapunov stability theory~(see condition CD2$^\prime$ in~\cite{MeyTwe93}), is proposed. Continuous-time Markov chains can be represented by the Chemical Master Equations (CMEs)~\cite{Kampen200796}, which are a set of ODEs.  When the system is sufficiently large, CMEs can be approximated by a set of SDEs called Chemical Langevin Equations (CLEs)~\cite{cite1079741} to which the method studied in this paper can be applied to find bounds.
\end{remark}

\subsection{Verifying Set Avoidance for PDE Systems}
 
Next, we provide a solution to Problem 1 for forward systems based on the construction of barrier functionals satisfying a set of inequalities. 
 
\begin{theorem}[Set Avoidance for Forward Systems] \label{thm:fwd}
Consider the forward PDE system described by~\eqref{eq:pde}. Let $u \in   \mathcal{U}_S(q)$. Given  a set of initial conditions $ \mathcal{U}_0 \subseteq  \mathcal{U}_S(q)$, an undesirable set~$ \mathcal{Y}_u$, such that $ \mathcal{U}_0 \cap  \mathcal{Y}_u = \emptyset$,  and a constant $t_0>0$, if there exists a barrier functional $B(t,u(t,x)) \in \mathcal{C}^1[t]$ as in \eqref{eq:barrierfunctional},
 such that the following inequalities hold
 \begin{subequations}
\label{eq:confwd}
\begin{multline} \label{con1}
B(t_0, u(t_0,x)) - B(0, u(0,x)) >0, \\ \forall u(t_0,x) \in  \mathcal{Y}_u,~\forall u_0 \in  \mathcal{U}_0
\end{multline}
\begin{equation} \label{con2}
 \frac{\mathrm{d}B(t,u(t,x))}{\mathrm{d}t} \le 0,~ \forall t\in[0,t_0],~\forall u \in  \mathcal{U},
\end{equation}
\end{subequations}
along the solutions of~\eqref{eq:pde}, then the solutions of~\eqref{eq:pde} avoid $ \mathcal{Y}_u$  at time $t_0$ (cf. Definition~\ref{def:safe}).
\end{theorem}
\begin{proof}
The proof is by contradiction. Assume there exists a solution of \eqref{eq:pde} such that, at time $t_0$, $u(t_0,x) \in  \mathcal{Y}_u$ and inequality~\eqref{con1} holds.  From \eqref{con2}, it follows that
\begin{equation} \label{eq1}
 \frac{\mathrm{d}  B(t,u(t,x))}{\mathrm{d}t}    \le 0,
\end{equation}
for all  $t\in[0,t_0]$, and  $u \in  \mathcal{U}$. Integrating both sides of~\eqref{eq1} with respect to $t$ from $0$ to $t_0$ yields
\begin{multline*} 
\int_0^{t_0} \frac{\mathrm{d}  B}{\mathrm{d}t} \mathrm{d}t =  B(t_0, u(t_0,x)) - B(0, u(0,x)) \le  0.
\end{multline*}
for all $u \in  \mathcal{U}$. This  contradicts \eqref{con1}.
\end{proof}

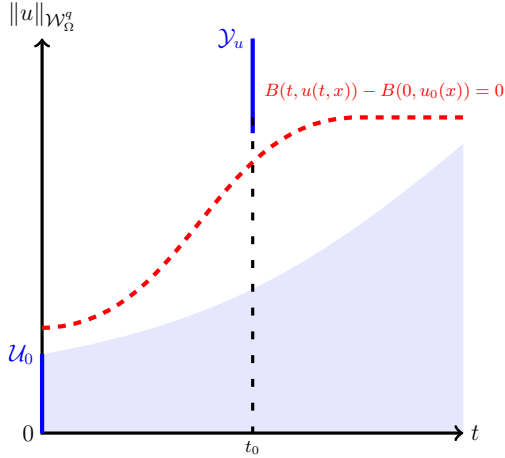
\begin{figure}
  \centering
\begin{tikzpicture}[scale=0.7, every node/.style={scale=0.75}]
\path[fill=blue!10] (0,0)--(0,1.5) to [out=10,in=220] (8,5.5) --(8,0);
\draw[very thick,<->] (8,0) node[font=\fontsize{13}{58}\sffamily\bfseries,right]{$t$} -- (0,0) node[font=\fontsize{13}{58}\sffamily\bfseries,left ]{$0$} -- (0,7.5) node[font=\fontsize{13}{58}\sffamily\bfseries,above]{$\| u \|_{ \mathcal{W}^q_{\Omega}}$};
\draw[dashed,red, ultra thick] (0,2) to [out=1,in=180] (6,6) --(8,6);
\draw[ultra thick,blue] (4,5.7)--(4,7.5) node[font=\fontsize{12}{58}\sffamily\bfseries,left ]{$ \mathcal{Y}_u$};
\draw[very thick,loosely dashed] (4,6)--(4,0) node[font=\fontsize{9}{58}\sffamily\bfseries,below ]{$t_0$};
\draw[ultra thick, blue] (0,0) --(0,1.5)node[font=\fontsize{12}{58}\sffamily\bfseries,left ]{$ \mathcal{U}_0$};
\draw [red] (6.5,6.8) node[font=\fontsize{9}{58}\sffamily\bfseries,below ]{$B(t,u(t,x))-B(0,u_0(x))=0$};
\end{tikzpicture}
  \caption{Illustration of a barrier functional for a forward PDE system: any solution $u(t,x)$ with  $u(0,x) \in \mathcal{U}_0$ (depicted by the shaded area) satisfies $u(t_0,x) \notin  \mathcal{Y}_u$. The system avoids $ \mathcal{Y}_{u}$ at time $t=t_0$ but not for  $\forall t>0$.}\label{tikzfig1}
\end{figure}

 \begin{remark}
 The level sets of~$B(t,u(t,x))-B(0, u_0(x))$  represent barrier surfaces in the $ \mathcal{U}$ space separating $ \mathcal{U}_0$ and $ \mathcal{Y}_u$ such that no solution of \eqref{eq:pde}  starting from $ \mathcal{U}_0$ enters $ \mathcal{Y}_u$ (hence, the term ``barrier functional''). This property is illustrated in Figure~\ref{tikzfig1}.
 \end{remark}

Theorem~\ref{thm:fwd} is concerned with conditions for set avoidance with respect to the undesirable set $ \mathcal{Y}_u$ at a particular time $t_0>0$. The next corollary follows from Theorem~\ref{thm:fwd} and gives conditions for set avoidance with respect to an undesirable set $ \mathcal{Y}_u$ for all time $t>0$.
\begin{corollary} \label{cor:fwd}
Consider the forward PDE system described by~\eqref{eq:pde}.  Assume  $u \in   \mathcal{U}_S(q)$. Given an undesirable set~$\mathcal{Y}_u \subset  \mathcal{Y}$, such that~$ \mathcal{U}_0 \cap  \mathcal{Y}_u = \emptyset$, if there exists a barrier functional $B(u(t,x))$ as in \eqref{eq:barrierfunctional}
 such that
 \begin{subequations}
 \begin{equation} \label{eq:cor:fwdcon1}
B(u(t,x)) - {B}(u_0(x)) > 0,~\forall u \in  \mathcal{Y}_u,~ \forall u_0 \in  \mathcal{U}_0,
 \end{equation}
 \begin{equation} \label{eq:cor:fwdcon2}
 \frac{\mathrm{d}B(u(t,x))}{\mathrm{d}t} \le 0,~\forall u \in  \mathcal{U},
 \end{equation}
 \end{subequations}
 along the solutions of \eqref{eq:pde}, then the solutions of forward PDE~\eqref{eq:pde} avoid $ \mathcal{Y}_u$ (cf. Definition~\ref{def:safe2}).
\end{corollary}

\begin{proof}
The proof follows the same lines as the proof of Theorem~\ref{thm:fwd}. Assume that there exists a solution $u(t,x)$ to~\eqref{eq:pde} such that, for some $t>0$, we have $u(t,x) \in \mathcal{Y}_u$. Then, from \eqref{eq:cor:fwdcon1}, it follows that $B(u(t,x)) - {B}(u_0(x)) > 0$. On the other hand, integrating inequality \eqref{eq:cor:fwdcon2} from $0$ to $t$ implies that $B(u(t,x)) - {B}(u_0(x)) \le 0$, which is a contradiction. Thus, since $t$ is arbitrary, the solutions to \eqref{eq:pde} avoid $\mathcal{Y}_u$ for all time.
\end{proof}

We conclude this section by illustrating Corollary~\ref{cor:fwd} with an example that uses the barrier formulation to bound a performance index. 
\begin{example}\textbf{(Performance Bounds)} Consider the heat equation defined over a domain $\Omega \subset \mathbb{R}^2$ 
\begin{equation} \label{eq:exampleheat}
\partial_t u = \Delta u,~~~~\mathrm{x} \in \Omega,~t>0,
\end{equation}
subject to $u |_{\partial \Omega} =0$ and
\begin{equation} \label{xcxz}
u(0,\mathrm{x}) \in  \mathcal{U}_0 = \left \{ u_0 \Big | \int_{\Omega} |\nabla u_0 |^2 \,\, \mathrm{d}\Omega \le 1 \right\}.
\end{equation}
The output mapping is given by
$$
y(t) =   \gamma^2 - \int_\Omega u^2\,\, \mathrm{d}\Omega,
$$
where $\gamma\geq0$. Then, the undesirable set is described as
$
 \mathcal{Y}_u =\left \{ u  \mid  y(t) = \gamma^2 - \int_\Omega u^2\,\, \mathrm{d}\Omega < 0 \right \}
$. We are interested in finding the minimum $\gamma$ such that no solution of\eqref{eq:exampleheat}  enters $ \mathcal{Y}_u$ for all $u(0,\mathrm{x}) \in  \mathcal{U}_0$.

We consider the  barrier functional~\eqref{eq:barrierfunctional} with
$$
\begin{array}{cc}
\mathscr{B}: & \mathcal{W}^1 \rightarrow \real_{\geq 0} \\ & u \mapsto \int_\Omega (\nabla u)^{T} \nabla u  \,\, \mathrm{d}\Omega,
\end{array}
$$
that is, $B(u(t,x))= \int_\Omega (\nabla u)^{T} \nabla u  \,\, \mathrm{d}\Omega$. We first check inequality \eqref{eq:cor:fwdcon2} along the solutions of \eqref{eq:exampleheat}:
\begin{align*} 
 \frac{\mathrm{d}  B(u(t,x))}{\mathrm{d}t}  &= \int_\Omega 2 \nabla u \partial_t \left( \nabla u  \right) \,\, \mathrm{d}\Omega \\ &= 2 \left(\nabla u \partial_t   u \right)|_{\partial \Omega} - 2 \int_\Omega \Delta u \partial_t   u  \,\, \mathrm{d}\Omega \\ &= -2 \int_\Omega  \left( \Delta u \right)^2  \,\, \mathrm{d}\Omega \le 0,
\end{align*}
where, in the second equality above, integration by parts and, in the third equality, the boundary conditions are used. Thus, inequality \eqref{eq:cor:fwdcon2} is satisfied. At this point, let us check inequality~\eqref{eq:cor:fwdcon1}. We have\begin{eqnarray} \nonumber
B(u(t,x))-B(u_0) &=&
\int_\Omega |\nabla u |^2 \,\, \mathrm{d}\Omega - \int_\Omega |\nabla u_0 |^2 \,\, \mathrm{d}\Omega \nonumber \\ &\ge& \int_\Omega |\nabla u  |^2  \,\, \mathrm{d}\Omega -1 \nonumber \\
 &\ge& C(\Omega) \int_\Omega  u^2\,\, \mathrm{d}\Omega -1, \nonumber
\end{eqnarray}
where $u_0 \in  \mathcal{U}_0$ as in~\eqref{xcxz} is applied to obtain the first inequality and Poincar\'{e} inequality~\cite{PW60} is used in the second inequality. Then, it follows that whenever $\gamma^2 > \frac{1}{C(\Omega)}$, we have
$
B(u(t,x))-B(u_0)>0,
$ 
and thus, from Theorem~\ref{thm:fwd}, system~\eqref{eq:exampleheat} avoids $ \mathcal{Y}_u$. Therefore, it holds that $y \notin  \mathcal{Y}_u$, which implies $y(t) = \gamma_{min}^2 - \int_\Omega u^2\,\, \mathrm{d}\Omega \ge 0$, i.e., $
\gamma_{min}^2 \ge \int_\Omega u^2\,\, \mathrm{d}\Omega,$ 
where $\gamma_{min}^2 = \frac{1}{C(\Omega)}$. For example, whenever $\Omega = \{ (x,y)\in \mathbb{R}^2 \mid |x+y|<1 \}$, we obtain $\gamma^2 =\frac{{2}}{\pi^2}$.
\end{example}

\section{Construction of Barriers Functionals} \label{sec:computational}

In this section, we study a specific class of barrier functionals. For the studied class and for particular sets $\mathcal{U}_0$, $\mathcal{U}_T$ and $\mathcal{Y}_u$,  the inequalities \eqref{eq:confwd} become integral inequalities. We then apply the method proposed in \cite{VAP16} to solve this class of inequalities to obtain the barrier functionals. For the case of polynomial data, the verification of the inequalities can be cast as constraints of an SDP. In this section, we abuse the notation and drop the dependence of $u$ on $t$ or/and $x$ for the sake of brevity

In the previous section, the barrier functionals were only assumed to be continuously differentiable with respect to time. In this section, we impose the following structure for the barrier functionals\begin{equation}
B(t,u) = \int_\Omega \zeta^d(D^{\alpha}u(t,\theta))^{T}\bar{B}(t,\theta)\zeta^d(D^{\alpha}u(t,\theta)) \mathrm{d} \theta
\label{eq:barrierquad}
\end{equation}
where $\Omega = [0,1]$ (note that any bounded domain $[a,b]$ on the real line can be mapped to $[0,1]$ with appropriate change of variables),  $\bar{B}: \real_{\geq 0}\times \Omega \rightarrow \real^{\sigma(n,d) \times \sigma(n,d)}$, $\bar{B}(t,x)\in \mathcal{C}^{1}[t], \forall x\in\Omega$, and the following quadratic-like structures for the undesirable and the initial sets
\begin{subequations}
\label{eq:quadsets}
\begin{equation}\label{eq:Yuquad}\mathcal{Y}_u =  \left\lbrace u \in \mathcal{U} \mid   \int_\Omega \zeta^{d}(D^{\alpha}u(\theta))^{T}Y(\theta) \zeta^d(D^{\alpha}u(\theta))   \mathrm{d}\theta \leq 0  \right\rbrace\end{equation} and the initial set 
\begin{equation}\mathcal{U}_0 =  \left\lbrace u_0 \in \mathcal{U} \mid  \int_\Omega \zeta^{d}(D^{\alpha}u(\theta))^{T}U_0(\theta) \zeta^d(D^{\alpha}u(\theta))  \mathrm{d}\theta \leq 0 \right\rbrace. \label{eq:U0quad}\end{equation}
\end{subequations}
where, $Y: \real_{\geq 0}\times \Omega \rightarrow \real^{\sigma(n,d) \times \sigma(n,d)}$ and  $U_0: \real_{\geq 0}\times \Omega \rightarrow \real^{\sigma(n,d) \times \sigma(n,d)}$.

We now present conditions for the verification of the barrier inequalities for sets as~\eqref{eq:quadsets}. 

Consider the following quadratic-like forms
\begin{equation*}
g_i(x) =  \zeta^{d}(D^{\alpha}u(x))^T G_i(x) \zeta^{d}(D^{\alpha}u(x)) 
\end{equation*}
$i  = 1,\ldots,r$  defining the set
\begin{equation}
\label{eq:setS}
\mathcal{S} = \left\lbrace u \in \mathcal{W}^{\alpha} \mid  \int_\Omega g_i(\theta)~\mathrm{d} \theta \leq 0,~i  = 1,\ldots,r     \right\rbrace.
\end{equation}

Note that for a given set  $\left\lbrace u \in \mathcal{W}^{\alpha} \mid  \int_\Omega h(\theta)~\mathrm{d} \theta = 0 \right\rbrace$ we can write  $g_1(x) = h(x)$, $g_2 = -h(x)$ to obtain the above representation.

Define 
\begin{subequations}
\label{eq:defV}
\begin{equation}
\label{eq:vi}
v_i(x):= \int_{0}^{x} g(\theta)~\mathrm{d} \theta,
\end{equation}
satisfying
\begin{align} 
v_i(0)  = 0, \label{eq:via} \\
\partial_x v_i(x)  -  g_i(x)  = 0 \label{eq:vib}
\end{align}
\end{subequations}
for $i = 1,\ldots,r$. Using~\eqref{eq:defV}, we represent the set~$\mathcal{S}$ as
$$\mathcal{S} = \left\lbrace u \in \mathcal{W}^{\alpha} \mid v_i(1)\leq 0,~i = 1,\ldots,r\right\rbrace.$$
 
 The following lemma is instrumental to cast the barrier inequalities as integral inequalities. Let $v(x) = \begin{bmatrix} v_1(x),\ldots,v_r(x)  \end{bmatrix}^T$ and $g(x)=\begin{bmatrix} g_1(x),\ldots,g_r(x)  \end{bmatrix}^T$.
 
 \begin{lemma} Given $R(t,u)$, if there exist a function vector $m: \mathcal{T}\times \Omega \rightarrow \real^{r}$ and a vector  $n \in \real^{r}_{\leq 0}$ such that
\begin{equation}
\label{eq:lemma1}
R(t,u) + n^{T}v(1)+ \int_\Omega m^{T}(t,\theta)  \left( \partial_\theta v(\theta)  -   g(\theta) \right)\mathrm{d} \theta > 0 
\end{equation}
$\forall t \in \mathcal{T}\subseteq \real_{\geq0}$, $\forall u \in \mathcal{U}$, then $R(t,u)\leq0$ for all $t \in \mathcal{T}$, for all $ u \in \mathcal{S}$.
\end{lemma}
\begin{proof}
From~\eqref{eq:vib}, we have, for any $m: \mathcal{T} \times \Omega \rightarrow \real^{r}$, 
\begin{equation*}m^{T}(t,x)\left( \partial_x v(x)  -   g(\theta) \right) = 0 \end{equation*}
for all $t \in \mathcal{T}$, $x\in \Omega$. Hence, for $v$ and $u$ related by \eqref{eq:vi} we have
\begin{equation*} \int_\Omega m^{T}(t,\theta)  \left( \partial_\theta v(\theta)  -  g(\theta) \right) \mathrm{d} \theta = 0.
\end{equation*}
and provided~\eqref{eq:lemma1} holds we have
\begin{equation*}
R(t,u) > -n^{T}v(1)
\end{equation*}
$\forall t \in \mathcal{T}$. Since $n^{T}v(1)\geq0$, for all $u \in \mathcal{S}$, we have that $R(t,u)\leq 0~\forall t \in \mathcal{T},~\forall u \in \mathcal{S}$.
\end{proof}

The following proposition applies the above lemma to formulate integral inequalities to verify the conditions of Theorem~\ref{thm:fwd} (a similar method can be applied to Theorem~\ref{thm:fwd2}) considering the barriers~\eqref{eq:barrierquad}, where we take  $\mathcal{S} = \mathcal{Y}_u \cap \mathcal{U}_0$, with the sets defined in~\eqref{eq:quadsets}, that is
\begin{align}
\label{eq:gprop}
g_1(x) &= \zeta^{d}(D^{\alpha}u(x))^{T}Y(x) \zeta^{d}(D^{\alpha}u(x)),\nonumber \\
g_2(x) &= \zeta^{d}(D^{\alpha}u(x))^{T}U_0(x) \zeta^{d}(D^{\alpha}u(x)). 
\end{align}

\begin{proposition}
\label{prop:intineqbarriers}
If there exist $\bar{B}: [0,t_0] \times \Omega \rightarrow \real^{\sigma(n,d) \times \sigma(n,d)}$, defining $B(t,u)$ as~\eqref{eq:barrierquad}, $m: \mathcal{T} \times \Omega \rightarrow \real^2$ and $n \in \real^{2}_{\leq0}$  such that the inequalities
\begin{subequations}
\begin{multline}
\left(B(t_0,u)  - B(0,u_0)  \right)   + n^{T}v(1) \\ + \int_{\Omega}^{} m^{T}(t,\theta)  \left( \partial_\theta v(\theta)  -   g(\theta) \right)\mathrm{d} \theta > 0,
\end{multline}
with $g(x) = \left[\begin{array}{cc} g_1(x) & g_2(x) \end{array} \right]^T$ defined by~\eqref{eq:gprop} and $v(x)=\begin{bmatrix} v_1(x) & v_2(x) \end{bmatrix}^T$ as defined by~\eqref{eq:defV}, and 
\begin{multline}
\label{eq:Bder}
 \int_\Omega \bigg( \zeta^d(D^{\alpha} u )^{T} \partial_t \bar{B}(t,\theta) \zeta^d(D^{\alpha} u ) \\ + 2 \zeta^d(D^{\alpha} u )^{T}  \bar{B}(t,\theta) \nabla \zeta^{d}(D^{\alpha} u )^{T} \partial_t u\bigg)~\mathrm{d}\theta\leq 0,
\end{multline}
\label{eq:barrier_poly}
\end{subequations}
$\forall t \in \left[0,t_0\right]$, $\forall u \in \mathcal{U}$,  then~\eqref{eq:confwd} holds. 
\end{proposition}

A method to solve integral inequalities as~\eqref{eq:barrier_poly}  has been proposed in \cite{VAP16} (also see~\cite{VAP15} for the formulation for $\Omega \subset \mathbb{R}^2$). In the proposed method, the problem of checking the integral inequality is cast as the problem of solving a differential linear matrix inequality. Such a formulation is possible thanks to the use of  quadratic-like expressions as in~\eqref{eq:barrierquad},~\eqref{eq:quadsets}.  Furthermore, in~\cite{VAP16} it is also shown how, for polynomial data, the corresponding differential matrix inequalities can be converted to Sum-of-Squares (SOS) program, which is then cast as an SDP. 

The numerical results presented in the next section consider the problem data to be polynomial, i.e. the functions $\bar{B}$, $m$, $Y$, $U_0$ appearing in the inequalities of Proposition~\ref{prop:intineqbarriers} are polynomials on variables $t$ and $x$, and the operator $\mathscr{F}$ in~\eqref{eq:pde} may be nonlinear and defined by a polynomial on $u$ and its spatial derivatives with coefficients that are polynomials on the spatial variables. The formulation of the SDPs can be automated and a plug-in to SOSTOOLS~\cite{valmorbida2015introducing} has been developed.

 
\begin{remark}  The assumptions on the set of terminal conditions $\mathcal{U}_T$ is the same as the assumptions on $\mathcal{U}_0$ for the computational formulation. In addition, for the set of terminal or initial conditions in the form of \eqref{t543}, we just need to substitute $u(t_0,x)=\phi(x)$ in the barrier functional, i.e.,
$B(t_0,u(t_0,x)) = \int_\Omega b(t_0,x,D^\beta \phi(x)) \,\, \mathrm{d}x.$ \end{remark}


\section{Examples} \label{sec:examples}

We now illustrate the proposed results with three numerical examples. The first example is associated with the option pricing problem from quantitative finance. The second example considers finding bounds on a functional of the states of an SDE describing biological stochastic resonance.  Lastly, the third example concerns a diffusion-reaction-convection PDE. 
The numerical results given in this section were obtained using  SOSTOOLS v. 3.00~\cite{PAVPSP13} and the associated SDPs were solved using SeDuMi v.1.02~\cite{Stu98} on a system with 2.5~GHz Intel Core~i5 and 16 GB of~RAM.

\subsection{Example 1: Option Pricing}

\begin{figure}[t]
\centerline{\includegraphics[scale=0.3]{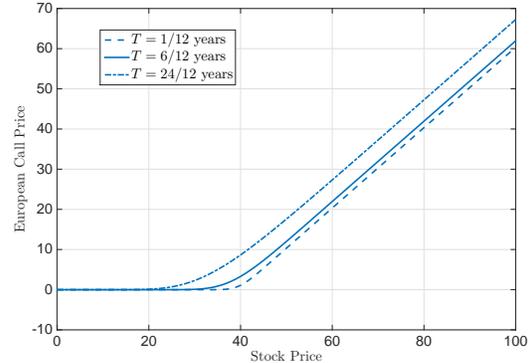}}
\caption{The solution to the Black-Scholes PDE~\eqref{eq:BS} at $t=0$ for different expiration times $T$ for a European call option.\label{fig3fsdf}}
\end{figure}

Consider the following SDE
\begin{equation}\label{eq:sdekllkkl}
\mathrm{d}S(\tau) = r S(\tau)\mathrm{d}s + \sigma S(\tau)\mathrm{d}W(\tau),~\tau\in [0,T],
\end{equation}
where $S$ is the spot price of the stock and $W$ is a Wiener process representing the uncertainty in the price history of the stock. The parameter $r\ge0$ denotes the risk-free interest rate and  $\sigma\ge0$ is the volatility of the stock. Let $\bar{s}$ be the maximum stock price. Applying  Theorem~\ref{thm:FK}, the value of an option maturing at time $T$ is defined as
\begin{equation*}
u(t,s) = e^{-r(T-t)} E\left[ f(S(T)) \mid S(t)=s   \right],
\end{equation*}
and satisfies the following backward PDE
\begin{multline} \label{eq:BS}
-\partial_t u(t,s) = \frac{\sigma^2 s^2}{2} \partial_s^2 u(t,s) + rs \partial_s u(t,s)-ru(t,s),\\~(t,s)\in [0,T]\times [0,\bar{s}],
\end{multline}
which is  the Black-Scholes equation for a non-dividend-paying stock (see also~\cite[p. 331]{Hu15}). For the European call option the terminal and the boundary conditions are given as
 \begin{equation*}
 \begin{cases}
 u(T,s)= f(s)=\max\left\{s-K, 0  \right\},\\
 u(t,0)= 0, \\
 u(t,\bar{s})=\bar{s},
 \end{cases}
 \end{equation*}
 where $K>0$ is the strike price. Assuming the stock is at-the-money, $f(s)=s-K$. The parameter values for a European call option~\cite[p. 338]{Hu15} are described as
 \begin{equation*}
 \begin{cases}
 T= 6/12~\text{(years)},~K=\$40,\\
 r=0.1,~\sigma=0.2.
 \end{cases}
 \end{equation*}
The maximum stock price is set to $\bar{s}=\$100$. The closed-form solution to~\eqref{eq:BS} is given by the Black-Scholes-Merton pricing formula~\cite{Mer73} (see \cite[p. 76]{9780511812545} for the derivation of the closed-form solution)
$$
u(t,s) = s N(d_1(t,s))-K e^{-r(T-t)}N(d_2(t,s)),
$$
where 
\begin{align}
d_1(t,s) &= \frac{\log(\frac{s}{K})+(r+\frac{\sigma^2}{2})(T-t)}{\sigma \sqrt{T-t}} \nonumber \\
d_2(t,s) &= d_1(t,s) -\sigma \sqrt{T-t}, \nonumber
\end{align}
and $N(\cdot)$ is the cumulative distribution function for a variable with a standard Gaussian distribution. Figure~\ref{fig3fsdf} illustrates the call option prices with respect to $s \in [0,100]$, which shows that the price of the option rises more when the option matures sooner. 

We are interested in finding an upper bound on the average option price $\frac{1}{\bar{s}} \int^{\bar{s}}_0 u(0,s) \,\, \mathrm{d}s$ without solving~\eqref{eq:sdekllkkl} or~\eqref{eq:BS}. To this end, we define
 $$
 \mathcal{Y}_u = \left\lbrace u \in \mathcal{L}^{1}_{[0,\bar{s}]}  \mid \frac{1}{\bar{s}} \int^{\bar{s}}_0 u(0,\theta) \,\, \mathrm{d}\theta  \ge \gamma   \right\rbrace,
 $$
 and we consider the following barrier functional
 \begin{equation*}
 B(t,u(t,x)) = \int_0^{\bar{s}} b(t,\theta) u^2(t,\theta) \,\, \mathrm{d}\theta.
 \end{equation*}
Using Theorem~\ref{thm:fwd2} and optimization problem~\eqref{eq:optm}, we obtain the upper bound $18.23$. As a matter of comparison, the actual upper bound for the average option price is $18.227$. The barrier functional certificate of degree $6$, given in Appendix B, was constructed in less than 3 seconds.

\subsection{Example 2: Stochastic Resonance in Biological Systems}

A switch-like response is observed in various signaling pathways
in biological systems. One method for modeling this behavior is by \textit{stochastic resonance}~\cite{Lon03}. Consider the following SDE
\begin{equation} \label{eq:sde22}
dX(\tau) = -\partial_X U(X(\tau)) \,\, \mathrm{d}\tau + \sigma(X(\tau)) \mathrm{d}W(\tau),~\tau \in [0,T]
\end{equation}
where $W(t)$ is a Wiener process, $U(X)=\frac{X^4}{4}-\frac{X^2}{2}$ is the double well potential, and $\sigma(X)= \sigma_0 \sqrt{1+X^2}$. System~\eqref{eq:sde22} with $dW=0$ has two stable equilibria at $X=+1,-1$ and an unstable equilibrium at $X=0$. Figure~\ref{fig3fsvdf} illustrates  five simulations of the system trajectories with initial condition~$X(0)=2$. The double well potential is common in nature. For instance, under certain circumstances, axons can function with two stable resting potentials. An axon is a nerve fiber projecting out of a neuron that conducts electrical impulses away from the neuron cell body. 

\begin{figure}[t]
\centerline{\includegraphics[scale=0.3]{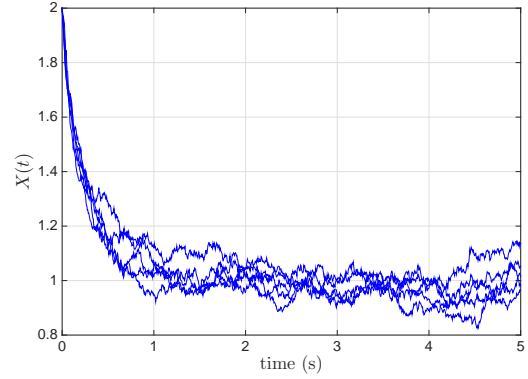}}
\caption{Five solutions to SDE~\eqref{eq:sde22} starting at $X(0)=2$.\label{fig3fsvdf}}
\end{figure}

Let $x \in \Omega = (0,5)$. Consider the following functional of the states 
\begin{multline*}
u(t,x) = E \left[  X^4(T)  + \int_t^T X^4(\theta) \,\, \mathrm{d}\theta \mid X(t)=x\right],
\end{multline*}
which is a functional of the $4$-th moment of the solutions to~\eqref{eq:sde22}. 
 Then, from Theorem~\ref{thm:FKBC}, $u(t,x)$ satisfies
\begin{multline*}
-\partial_t u(t,x) = \frac{\sigma_0^2(1+x^2)}{2} \partial_x^2 u(t,x) \\+x(1-x^2)\partial_xu(t,x)+x^4,~t \in [0,T],~x\in \Omega.
\end{multline*}
subject to the following terminal and boundary conditions
\begin{equation*}
u(T,x)=x^4,~u(t,0)=0,~u(t,5)=5^4.
\end{equation*}
We are interested in finding an upper bound to 
\begin{equation} \label{eq:xcxsdsd}
u(0,4) = E \left[  X^4(T) + \int_0^T X^4(\theta) \,\, \mathrm{d}\theta \mid X(0)=4\right],
\end{equation}
with $T=5$ seconds. To this end, we consider 
$$
B(t,u) = \int_0^5 b(t,\theta) u^2(t,\theta)\,\, \mathrm{d}\theta.
$$
The obtained bound using the proposed method is~257.89. The value for~\eqref{eq:xcxsdsd} computed by Monte Carlo simulations and trapezoidal integration was 256.4851 (see Figure~\ref{fig3rrefsvdf} for the time evolution of~$u(t,4)$). The constructed certificate of degree 8 computed in less than $8$ seconds is given in Appendix B. 
\begin{figure}[t]
\centerline{\includegraphics[scale=0.3]{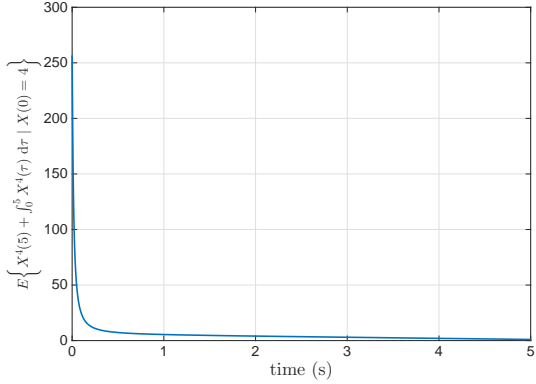}}
\caption{The evolution of $u(t,4)$ with respect to time.\label{fig3rrefsvdf}}
\end{figure}
\subsection{Example 3: Diffusion-Reaction-Convection PDE}

Consider 
\begin{equation} \label{eq:examplemain}
\partial_t u = \partial_x^2 u + \lambda u - 2 u \partial_x u,~u(t,0)=u(t,1)=0,
\end{equation}
where $\lambda>0$, $x \in (0,1)$ and $t>0$.  Due to the presence of a nonlinear convection term, the solutions with $\lambda \ge \pi^2$ (otherwise unstable) may converge to a different stationary solution. This stems from the fact that the convection term transfers low wave number components of the solutions to the high wave number ones for which the diffusion term has a stabilizing effect in a similar fashion to the effects of diffusion and anti-diffusion terms in the Kuramoto-Sivashinsky equation~\cite{Ott09}. Figure~\ref{fig3} depicts a solution to PDE~\eqref{eq:examplemain} with $\lambda > \pi^2$.

We are interested in computing the maximum value for parameter $\lambda$, such that the solutions starting in
\begin{equation} \label{eq:u00}
 \mathcal{U}_0 = \left\lbrace u_0   \mid \int_0^1 \left( u_0^2 + (\partial_\theta u_0)^2  \right) \,\, \mathrm{d}\theta \le 1  \right\rbrace,
\end{equation}
which implies $\| u_0 \|_{\mathcal{W}^1_{(0,1)}} \le 1$, do not enter the set
$ \mathcal{Y}_u = \left\lbrace u   \mid \int_0^1 \left( u^2 + (\partial_\theta u)^2  \right) \,\, \mathrm{d}\theta \ge (6)^2   \right\rbrace,$
i.e., $\| u \|_{\mathcal{W}^1_{(0,1)}} \ge {6}$ for all $t>0$. To this end, we consider the following barrier functional structure
\begin{equation} \label{bbbasjka}
B(t,u(t,x)) = \int_0^1 \left[\begin{smallmatrix} u(t,\theta) \\ \partial_\theta u(t,\theta)  \end{smallmatrix}\right]^{T}  M(\theta) \left[\begin{smallmatrix} u\theta \\ \partial_\theta u(t,\theta)  \end{smallmatrix}\right] \,\, \mathrm{d}\theta,
\end{equation}
where $M(\theta) \in \mathbb{R}^{2\times 2}$. Applying Corollary~\ref{cor:fwd} and performing a line search for $\lambda$, the maximum parameter $\lambda$, for which the solutions avoid $ \mathcal{Y}_u$, is found to be $\lambda = 1.195\pi^2$, for which the barrier functional~\eqref{bbbasjka} was constructed with a degree-16 $M(\theta)$ as given in Appendix B less than 16 seconds. This is consistent with the numerical experiments  shown in Figure~\ref{fig2}, where the $\mathcal{W}^1$-norm of the solution to PDE~\eqref{eq:examplemain} with $\lambda=1.2\pi^2$ was computed  for four different initial conditions $u_0(x) \in  \mathcal{U}_0$ as in~\eqref{eq:u00}.
\begin{figure}[!tb]
\begin{psfrags}
     \psfrag{t}[l][l]{\footnotesize $t$}
     \psfrag{x}[l][l]{\scriptsize $x$}
     \psfrag{u}[c][l][1][-90]{\footnotesize $u(t,x)$}
\epsfxsize=7.5cm
\centerline{\epsffile{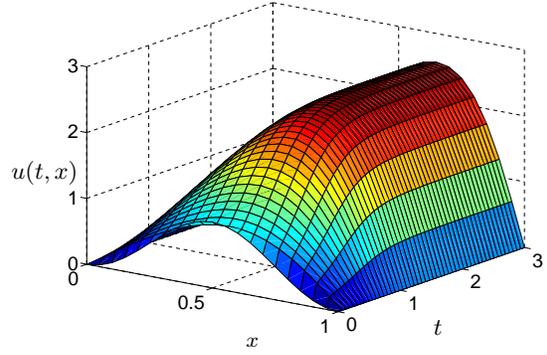}}
\end{psfrags}
\caption{The solution to PDE~\eqref{eq:examplemain} for $\lambda=1.2\pi^2$.\label{fig3}}
\end{figure}

\begin{figure}[!tb]
\begin{psfrags}
     \psfrag{t}[l][l]{\footnotesize $t$}
     \psfrag{u100000}[l][l]{\scriptsize $u_0 = 2x^2(1-x)^2$}
     \psfrag{u2}[l][l]{\scriptsize $u_0 = 0.4x(e^x - e)$}
     \psfrag{u3}[l][l]{\scriptsize $u_0 = x(1-x)$}
     \psfrag{u4}[l][l]{\scriptsize $u_0 = \ln (\frac{11}{20}x(1-x)+1)$}
     \psfrag{u}[c][l][1][-90]{\footnotesize $\|u\|_{\mathcal{W}^{1}_{(0,1)}}~~~$}
\epsfxsize=7.5cm
\centerline{\epsffile{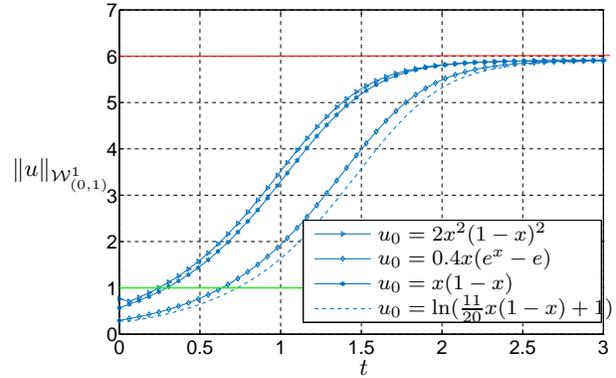}}
\end{psfrags}
\caption{The evolution of $\mathcal{W}^1_{(0,1)}$-norm of solutions to \eqref{eq:examplemain} with  $\lambda=1.2\pi^2$ for  different initial conditions. The red and green lines show the boundaries of $ \mathcal{Y}_{u}$ and $ \mathcal{U}_0$, respectively.\label{fig2}}
\end{figure}
 
 \section{Conclusion}\label{sec:conclusion}

We proposed a method based on barrier functionals to address  analysis problems of complex dynamical systems. For SDEs, we presented a method for bounding functionals of the states thanks to the Feynman-Kac formula. For PDEs, we developed a  method for verifying whether the solutions of the PDE would avoid an undesirable set. Numerical examples illustrate the computation of barrier functional certificates by SDPs for problems with polynomial data and equations in one-dimensional spatial domain.  The extension of the results to systems in two-dimensional domains is under study and preliminary results were presented in~\cite{VAP15}.



\bibliographystyle{plain}        
\bibliography{references}           



\appendix
\section{Well-posedness of PDE Systems}
%
 \renewcommand{\theequation}{A.\arabic{equation}}
  \setcounter{equation}{0}  
  
  
We briefly review aspects related to the on well-posedness of PDEs.In the case where $\mathscr{F}$ is a linear operator, the well-posedness problem of \eqref{eq:pde} is tied to ${F}$ being the generator of a strongly continuous semigroup denoted $C^0$-Semigroup~\cite[Chapter 2.1]{CZ95}. In this respect, the Hille-Yosida theorem~\cite[Theorem 3.4.1]{Sta05}, \cite[Theorem 2.1.12]{CZ95} provides necessary and sufficient conditions for such generators. In addition, given an operator, the Lumer-Phillips theorem~\cite[Theorem 3.4.5]{Sta05},\cite{lumer1961},\cite[Theorem 3.8.6]{TW09} presents conditions for the generator of a strongly continuous semigroup  that are easier to verify based on checking whether the operator is dissipative.

If $\mathscr{F}$ is a nonlinear dissipative operator satisfying
$$
Dom(\mathscr{F}) \subset Ran(\mathscr{I}-\lambda \mathscr{F}),~~~~\forall \lambda >0,
$$
with $\mathscr{I}$ representing the identity operator, then $\mathscr{F}$ generates a (nonlinear) semigroup of contractions~\cite[Corollary 2.10]{Mi92}. In addition, uniqueness and existence of the solutions to \eqref{eq:pde}  follows from~\cite[Theorem 4.10 and Theorem 5.1]{Mi92}.

\section{Numerical Results}
 
Neglecting the terms with coefficients smaller than $10^{-4}$, the constructed certificate for Example 1 is given by\begin{multline}
10^4b(t,\theta) = -7.916\theta^6+105.7\theta^5t+ 195.0\theta^5-315.15\theta^4t^2 \\ 
+175.7\theta^4t-348.2\theta^4-35.99\theta^3t^3-26.33\theta^3t \\
-72.06\theta^3+42.64\theta^2t^3-66.52\theta^2t^2+203.8\theta^2t\\
-228.9\theta^2-2.782\theta t^5-4.065\theta t^4-228.9\theta^2\\
-2.782\theta t^5-4.065\theta t^4-1.184\theta t^2+2.485\theta t\\
-15.97\theta-631.9t^6+62.17t^5-162.0t^4\\
+230.8t^3-59.17t^2+717.7t-705.7. \nonumber
\end{multline}
The constructed certificate for Example 2 is 
\begin{multline}
10^4 b(t,\theta) = 1.794t^7\theta + 2.789t^7  - 3.187t^6\theta - 6.006t^6 \\
+ 5.092t^5\theta^3 - 1.344t^5\theta^2 - 3.984t^5\theta + 3.039t^5 \\
 - 7.186t^4\theta^3 - 2.41t^4\theta^2 + 8.11t^4\theta + 3.204t^4 \\
  - 2.429t^3\theta^5 + 5.783t^3\theta^4 + 1.152t^3\theta^3 + 2.905t^3\theta^2\\
  - 2.765t^3 - 2.398t^2\theta^5 - 1.27t^2\theta^4 + 1.315t^2\theta^3 \\
  + 4.31t^2\theta^2 - 2.42t^2\theta - 1.086t^2 - 1.061t\theta^5 \\
  + 6.986t\theta^4 - 5.973t\theta^2  + 2.757t + 1.508\theta^6 \\
   + 1.889\theta^2 -1.895. \nonumber
\end{multline}
Neglecting the terms with coefficients smaller than $10^{-4}$, the constructed certificate for Example 3 is given by
\begin{equation*}
M(\theta) = \left[\begin{smallmatrix} M_{11}(\theta) & M_{12}(x) \\ M_{12}(\theta)  & M_{22}(\theta)   \end{smallmatrix}\right],
\end{equation*}
\begin{multline}
10^{4}M_{11}(\theta) = - 12.96\theta^{16} + 27.92\theta^{15} - 55.38\theta^{14} \\- 160.6\theta^{13}
 - 222.4\theta^{12} + 180.8\theta^{11} + 199.1\theta^{10} \\+ 332.9\theta^9
  - 343.5\theta^8 - 454.9\theta^7 - 390.1\theta^6 \\+ 329.9\theta^5
 + 666.7\theta^4 - 83.37\theta^3 - 663.4\theta^2 \\+ 418.7\theta - 74.97, \nonumber
\end{multline}
\vspace{-3mm}
\begin{multline}
10^{4}M_{12}(\theta) =  1.39\theta^{16} - 26.03\theta^{15} + 10.76\theta^{14} \\
+ 22.53\theta^{13} - 14.63\theta^{12} - 22.81\theta^{11} + 52.28\theta^{10}\\
 - 67.56\theta^9 - 69.45\theta^8 - 87.54\theta^7 + 79.37\theta^6 \\
 + 262.8\theta^5 - 32.63\theta^4 - 447.1\theta^3 + 417.7\theta^2 \\
 - 157.6\theta + 23.88, \nonumber
 \end{multline}
 \begin{multline}
10^{4}M_{22}(\theta) =  - 1.607\theta^{16}  - 26.85\theta^{14} + 47.17\theta^{13} \\
+ 38.69\theta^{12} - 77.1\theta^{11} - 34.36\theta^{10} + 66.47\theta^9 \\
+ 13.36\theta^8 - 34.57\theta^7 - 1.477\theta^6 + 17.13\theta^5 \\
- 9.405\theta^4 + 2.768\theta^3. \nonumber
 \end{multline}
\normalsize

\end{document}